\documentclass[12pt,a4paper,twoside,english,american]{scrartcl}
\usepackage{helvet}
\usepackage[T1]{fontenc}
\usepackage[latin9]{inputenc}
\usepackage{amsmath}
\usepackage{amssymb}

\makeatletter

 \usepackage{amsthm}
 \theoremstyle{plain}
\newtheorem{thm}{Theorem}[section]
  \theoremstyle{plain}
  \newtheorem{lem}[thm]{Lemma}
  \theoremstyle{definition}
  \newtheorem{defn}[thm]{Definition}
  \theoremstyle{remark}
  \newtheorem{rem}[thm]{Remark}
  \theoremstyle{plain}
  \newtheorem{prop}[thm]{Proposition}

\usepackage[T1]{fontenc}    % Silbentrennung auch nach Umlauten mglich

\usepackage{a4wide}        % A4-Format
\usepackage{tu-preprint}
\usepackage{amsmath}

\usepackage{euscript}
\usepackage{mathabx}
\allowdisplaybreaks[1] % allow page breaks in some displayed equations
\usepackage{amsfonts}
\newenvironment{keywords}{ \noindent\footnotesize\textbf{Keywords and phrases:}}{}

\newenvironment{class}{\noindent\footnotesize\textbf{Mathematics subject classification 2000:}}{}

\usepackage{color,tu-preprint}

\newcommand*{\dive}{\operatorname{div}}

\newcommand*{\grad}{\operatorname{grad}}

\renewcommand*{\i}{\mathrm{i}}
\DeclareMathAccent{\Circ}{\mathalpha}{operators}{"17}
\newcommand{\interior}[1]{\Circ{#1}}
\renewcommand{\Im}{\operatorname{\mathfrak{Im}}}
\renewcommand{\Re}{\operatorname{\mathfrak{Re}}}

\newcommand*{\conv}[1]{\overset{#1\to \infty}{\to }}

\renewcommand*{\epsilon}{\varepsilon}
\renewcommand*{\rho}{\varrho}

\makeatother

\usepackage{babel}
\begin{document}
\selectlanguage{english}%
\institut{Institut f\"ur Analysis}

\preprintnumber{MATH-AN-08-2012}

\preprinttitle{A Class of Evolutionary Problems with an Application to Acoustic Waves with Impedance Type Boundary Conditions.}

\author{Rainer Picard}

\makepreprinttitlepage

\selectlanguage{american}%
\setcounter{section}{-1}

\date{}

\title{A Class of Evolutionary Problems with an Application to Acoustic
Waves with Impedance Type Boundary Conditions.}

\author{Rainer Picard\\
Institut für Analysis,Fachrichtung Mathematik\\
 Technische Universität Dresden\\
 Germany\\
 rainer.picard@tu-dresden.de\\
\\
January 7, 2011  }
\maketitle
\begin{abstract}
A class of evolutionary operator equations is studied. As an application
the equations of linear acoustics are considered with complex material
laws. A dynamic boundary condition is imposed which in the time-harmonic
case corresponds to an impedance or Robin boundary condition. Memory
and delay effects in the interior and also on the boundary are built
into the problem class.
\end{abstract}

%\begin{class}
%XXXX
%\end{class}
%evolution equations, partial differential equations, causality, impedance, memory, delay   \end{keywords}
\begin{class}  
{Primary 35F05; Secondary 35L40, 35F10, 76Q05}
\end{class}
%\keywords{Evolution equations, partial differential equations, acoustic waves, causality, impedance type boundary condition, memory, delay}

\textbf{Dedicated to Prof. Dr. Rolf Leis on the occasion of his 80th
birthday.} 

\tableofcontents{}

\section{Introduction}

In  \cite{Pi2009-1} and \cite{PDE_DeGruyter}, chapter 5, a theoretical
framework has been presented to discuss typical linear evolutionary
problems as they arise in various fields of applications. The suitability
of the problem class described for such applications has been demonstrated
by numerous examples of varying complexity. The problem class can
be heuristically described as: finding $U,V$ satisfying
\[
\partial_{0}V+AU=f,
\]
where $V$ is linked to $U$ by a linear material law 
\[
V=\mathbb{M}U.
\]
Here $\partial_{0}$ denotes the time derivative, $\mathbb{M}$ is
a bounded linear operator commuting with $\partial_{0}$ and $A$
is a (usually) unbounded translation invariant linear operator. The
material law operator $\mathbb{M}$ is given in terms of a suitable
operator valued function of the time derivative $\partial_{0}$ in
the sense of a functional calculus associated with $\partial_{0}$
realized as a normal operator. The focus in the quoted references
is on the case where $A$ is skew-selfadjoint. 

The aim of this paper is to extend the general theory to encompass
an even larger class of problems by allowing $A$ to be more general.
The presentation will rest on a conceptually more elementary version
of this theory as presented in \cite{2009-2}, where the above problem
is discussed as establishing the continuous invertibility of the unbounded
operator sum $\partial_{0}\mathbb{M}+A$, i.e. we shall develop a
solution theory for a class of operator equations of the form $\overline{\partial_{0}\mathbb{M}+A}\: U=f.$
For suggestiveness of notation we shall simply write $\partial_{0}\mathbb{M}+A$
instead of $\overline{\partial_{0}\mathbb{M}+A}$, which can indeed
be made rigorous in a suitable distributional sense. 

To exemplify the utility of the generalization we shall apply the
ideas developed here to impedance type boundary conditions in linear
acoustics.

After briefly describing the corner stones of a general solution theory
in section \ref{sec:General-Solution-Theory} we shall discuss in
section \ref{sec:Causality-and-Memory} the particular issue of causality,
which is a characteristic feature of problems we may rightfully call
\emph{evolutionary}.

The general findings will be illustrated by an application to to acoustic
equations with a dynamic boundary condition allowing for additional
memory effects on the boundary of the underlying domain. We refer
to the boundary condition as of \emph{impedance type} due to its form
after Fourier-Laplace transformation with respect to time. The reasoning
in \cite{PIC_2010:1889} finds its generalization in the arguments
presented here in so far as here evolutionary boundary conditions
modelling a separate dynamics on the boundary are included.

\section{\label{sec:General-Solution-Theory}General Solution Theory}

First we specify the space and the class of material law operators
we want to consider. 

\textbf{Assumptions on the material law operator}: Let $M=\left(M\left(z\right)\right)_{z\in B_{\mathbb{C}}\left(r,r\right)}$
be a family of uniformly bounded linear operators in a Hilbert space
$H$ holomorphic in the ball $B_{\mathbb{C}}\left(r,r\right)$ of
radius $r$ centered at $r$. Then, for $\rho>\frac{1}{2r}$, we define
\begin{align*}
\mathbb{M}\;:=M\left(\partial_{0}^{-1}\right) & ,
\end{align*}
where 
\[
M\left(\partial_{0}^{-1}\right)\;:=\mathbb{L}_{\rho}^{*}M\left(\frac{1}{\mathrm{i}\, m_{0}+\rho}\right)\mathbb{L}_{\rho}.
\]

Here $\mathbb{L}_{\rho}:H_{\rho}\left(\mathbb{R},H\right)\to H_{0}\left(\mathbb{R},H\right)$,
$\rho\in\mathbb{R}_{\geq0}$, denotes the unitary extension of the
Fourier-Laplace transform to $H_{\rho}\left(\mathbb{R},H\right)$,
the space of $H$-valued $L^{2,\mathrm{loc}}$-functions $f$ on $\mathbb{R}$
with
\[
\left|f\right|_{\rho,0,0}\;:=\sqrt{\int_{\mathbb{R}}\left|f\left(t\right)\right|_{H}^{2}\,\exp\left(-2\rho t\right)\: dt}<\infty.
\]
The Fourier-Laplace transform is given by
\[
\left(\mathbb{L}_{\rho}\varphi\right)\left(s\right)=\frac{1}{\sqrt{2\pi}}\,\int_{\mathbb{R}}\exp\left(-\i\left(s-\i\rho\right)t\right)\:\varphi\left(t\right)\: dt,\: s\in\mathbb{R},
\]
for $\varphi\in\interior C_{\infty}\left(\mathbb{R},H\right)$, i.e.
for smooth $H$-valued functions $\varphi$ with compact support.
The multiplicatively applied operator $M\left(\frac{1}{\mathrm{i}\, m_{0}+\rho}\right):H_{0}\left(\mathbb{R},H\right)\to H_{0}\left(\mathbb{R},H\right)$
is given by
\[
\left(M\left(\frac{1}{\mathrm{i}\, m_{0}+\rho}\right)\varphi\right)\left(s\right)=M\left(\frac{1}{\mathrm{i}\, s+\rho}\right)\varphi\left(s\right),\: s\in\mathbb{R},
\]
for $\varphi\in\interior C_{\infty}\left(\mathbb{R},H\right)$, (so
that $m_{0}$ simply denotes the multiplication by the argument operator). 

~

$H_{\rho}\left(\mathbb{R},H\right)$ is a Hilbert space with norm
$\left|\:\cdot\:\right|_{\rho,0,0}$. The associated inner product,
assumed to be linear in the \emph{second} factor, will be denoted
by $\left\langle \:\cdot\:|\:\cdot\:\right\rangle _{\rho,0,0}$. In
the case $\rho=0$ the space $H_{\rho}\left(\mathbb{R},H\right)$
is simply the space $L^{2}\left(\mathbb{R},H\right)$ of $H$-valued
$L^{2}$-functions on $\mathbb{R}$. In our general framework there
is, however, a bias to consider large $\rho\in\mathbb{R}_{>0}$.

Note that for $r\in\mathbb{R}_{>0}$
\begin{align*}
B_{\mathbb{C}}\left(r,r\right) & \to\left[\mathrm{i}\,\mathbb{R}\right]+\left[\mathbb{R}_{>1/\left(2r\right)}\right]\\
z & \mapsto z^{-1}
\end{align*}
is a bijection. In $H_{\rho}\left(\mathbb{R},H\right)$ the closure
$\partial_{0}$ of the derivative on $\interior C_{\infty}\left(\mathbb{R},H\right)$
turns out to be a normal operator, see e.g. \cite{2009-2}, with $\Re\left(\partial_{0}\right)=\rho$. 

With the time translation operator $\tau_{h}:H_{\rho}\left(\mathbb{R},H\right)\to H_{\rho}\left(\mathbb{R},H\right)$,
$h\in\mathbb{R}$, given by
\[
\left(\tau_{h}\varphi\right)\left(s\right)=\varphi\left(s+h\right),\: s\in\mathbb{R},
\]
for $\varphi\in\interior C_{\infty}\left(\mathbb{R},H\right)$ it
is easy to see that $M\left(\partial_{0}^{-1}\right)$ is translation
invariant, i.e.
\[
\tau_{h}M\left(\partial_{0}^{-1}\right)=M\left(\partial_{0}^{-1}\right)\tau_{h}\,,\: h\in\mathbb{R}.
\]
This is indeed clear since $M\left(\partial_{0}^{-1}\right)$ commutes
by construction with $\partial_{0}^{-1}$ and 
\[
\tau_{h}=\exp\left(h/\partial_{0}^{-1}\right).
\]

\textbf{Assumptions on $A$:} For the densely defined, closed linear
operator $A$ in $H_{\rho}\left(\mathbb{R},H\right)$ we also assume
commutativity with $\partial_{0}^{-1}$ which implies commutativity
with bounded Borel functions of $\partial_{0}^{-1}$, in particular,
translation invariance
\[
\tau_{h}A=A\tau_{h}\:,\: h\in\mathbb{R}.
\]

We shall use $M\left(\partial_{0}^{-1}\right)$ and $A$ without denoting
a reference to $\rho\in\mathbb{R}_{>0}$. 

Thus we are led to solving 
\[
\left(\partial_{0}M\left(\partial_{0}^{-1}\right)+A\right)U=f
\]
in $H_{\rho}\left(\mathbb{R},H\right)$ with $f\in H_{\rho}\left(\mathbb{R},H\right)$
given. Recall that we have chosen to write $\partial_{0}M\left(\partial_{0}^{-1}\right)+A$
for the closure of sum of the two discontinuous operators involved.
Note that initial data are not explicitely prescribed. It is assumed
that they are built into the source term $f$ so that vanishing initial
data may be assumed. 

\textbf{Condition (Positivity 1\label{Condition-(Positivity I)}):}

\noindent \begin{center}
\begin{minipage}[t]{0.98\columnwidth}%
\textit{For all $U\in D\left(\partial_{0}\right)\cap D\left(A\right)$
and $V\in D\left(\partial_{0}\right)\cap D\left(A^{*}\right)$ we
have uniformly for all sufficiently large $\rho\in\mathbb{R}_{>0}$
some $\beta_{0}\in\mathbb{R}_{>0}$ such that }
\[
\Re\left\langle \chi_{_{\mathbb{R}\leq0}}\left(m_{0}\right)U|\left(\partial_{0}M\left(\partial_{0}^{-1}\right)+A\right)U\right\rangle _{\rho,0,0}\geq\beta_{0}\left\langle \chi_{_{\mathbb{R}\leq0}}\left(m_{0}\right)U|U\right\rangle _{\rho,0,0},
\]
\[
\Re\left\langle V|\left(\partial_{0}^{*}M^{*}\left(\left(\partial_{0}^{-1}\right)^{*}\right)+A^{*}\right)V\right\rangle _{\rho,0,0}\geq\beta_{0}\left\langle V|V\right\rangle _{\rho,0,0}.
\]
\end{minipage}
\par\end{center}

Here we have used the notation
\[
M^{*}\left(z\right)\::=M\left(z^{*}\right)^{*}
\]
with which we have 
\begin{eqnarray*}
M\left(\partial_{0}^{-1}\right)^{*} & = & M^{*}\left(\left(\partial_{0}^{-1}\right)^{*}\right).
\end{eqnarray*}

Note that due to translation invariance \textbf{Condition (Positivity
\ref{Condition-(Positivity I)})} is equivalent to \textbf{Condition
(Positivity 2\label{Condition-(Positivity II)})}%
\footnote{\noindent This is the assumption employed in \cite{PIC_2010:1889}.
The second inequality in the corresponding assumption in \cite{PIC_2010:1889}
erroneously also contains the cut-off multiplier $\chi_{_{\mathbb{R}\leq a}}\left(m_{0}\right)$
which should not be there and is indeed never utilized in the arguments.%
}:

\noindent \begin{center}
\begin{minipage}[t]{0.98\columnwidth}%
\textit{For all $a\in\mathbb{R}$ and all $U\in D\left(\partial_{0}\right)\cap D\left(A\right)$,
$V\in D\left(\partial_{0}\right)\cap D\left(A^{*}\right)$ we have,
uniformly for all sufficiently large $\rho\in\mathbb{R}_{>0}$, some
$\beta_{0}\in\mathbb{R}_{>0}$ such that }
\[
\Re\left\langle \chi_{_{\mathbb{R}_{\leq a}}}\left(m_{0}\right)U|\left(\partial_{0}M\left(\partial_{0}^{-1}\right)+A\right)U\right\rangle _{\rho,0,0}\geq\beta_{0}\left\langle \chi_{_{\mathbb{R}\leq a}}\left(m_{0}\right)U|U\right\rangle _{\rho,0,0},
\]
\[
\Re\left\langle V|\left(\partial_{0}^{*}M^{*}\left(\left(\partial_{0}^{-1}\right)^{*}\right)+A^{*}\right)V\right\rangle _{\rho,0,0}\geq\beta_{0}\left\langle V|V\right\rangle _{\rho,0,0}.
\]
\end{minipage}
\par\end{center}

Here $\chi_{_{M}}$ denotes the characteristic function of the set
$M$. As a general notation we introduce for every measurable function
$\psi$ the associated multiplication operator 
\[
\psi\left(m_{0}\right):H_{\rho}\left(\mathbb{R},H\right)\to H_{\rho}\left(\mathbb{R},H\right)
\]
determined by
\[
\left(\psi\left(m_{0}\right)f\right)\left(t\right)\;:=\psi\left(t\right)\: f\left(t\right),\: t\in\mathbb{R},
\]
for $f\in\interior C_{\infty}\left(\mathbb{R},H\right)$.Letting $a$
go to $\infty$ in this leads to \textbf{Condition (Positivity 3\label{Condition-(Positivity III)})}:

\noindent \begin{center}
\begin{minipage}[t]{0.98\columnwidth}%
\textit{For all $U\in D\left(\partial_{0}\right)\cap D\left(A\right)$
and $V\in D\left(\partial_{0}\right)\cap D\left(A^{*}\right)$ we
have, uniformly for all sufficiently large $\rho\in\mathbb{R}_{>0}$,
some $\beta_{0}\in\mathbb{R}_{>0}$ such that }
\[
\Re\left\langle U|\left(\partial_{0}M\left(\partial_{0}^{-1}\right)+A\right)U\right\rangle _{\rho,0,0}\geq\beta_{0}\left\langle U|U\right\rangle _{\rho,0,0},
\]
\[
\Re\left\langle V|\left(\partial_{0}^{*}M^{*}\left(\left(\partial_{0}^{-1}\right)^{*}\right)+A^{*}\right)V\right\rangle _{\rho,0,0}\geq\beta_{0}\left\langle V|V\right\rangle _{\rho,0,0}.
\]
\end{minipage}
\par\end{center}

This is the constraint we have used in previous work. In the earlier
considered cases the seemingly stronger \textbf{Condition (Positivity
\ref{Condition-(Positivity II)})} can be shown to hold. It turns
out, however, that in the translation invariant case \textbf{Condition
(Positivity \ref{Condition-(Positivity I)})} adds flexibility to
the solution theory (to include more general cases) and makes the
issue of causality more easily accessible. 

In preparation of our well-posedness result we need the following
lemma. Note that for sake of clarity here we do distinguish between
the natural sum $A+B$ and its closure.
\begin{lem}
Let $A+B$ and $A^{*}+B^{*}$ be densely defined. $\left(P_{n}\right)_{n\in\mathbb{N}}$
a monotone sequence of orthogonal projectors commuting with $A$ and
$B$ with $P_{n}\conv{n}1$ strongly, such that $\left(P_{n}BP_{n}\right)_{n\in\mathbb{N}}$
is a sequence of con\-ti- nu\-ous linear operators. Then\vspace{-0.3cm}
\[
\left(P_{n}AP_{n}\right)^{*}=P_{n}A^{*}P_{n},\:\left(P_{n}BP_{n}\right)^{*}=P_{n}B^{*}P_{n}
\]
 \vspace{-0.3cm}for every $n\in\mathbb{N}$. Moreover,
\end{lem}
\[
\overline{A^{*}+B^{*}}=\left(A+B\right)^{*}=\mathrm{s-}\lim_{n\to\infty}\left(P_{n}\left(A+B\right)P_{n}\right)^{*}.
\]

\begin{proof}
We have $\left\langle Ax|y\right\rangle _{H}=\left\langle x|A^{*}y\right\rangle _{H}$
for all $x\in D\left(A\right)$ and so also for $n\in\mathbb{N}$
\[
\left\langle Ax|P_{n}y\right\rangle _{H}=\left\langle P_{n}Ax|y\right\rangle _{H}=\left\langle AP_{n}x|y\right\rangle _{H}=\left\langle P_{n}x|A^{*}y\right\rangle _{H}=\left\langle x|P_{n}A^{*}y\right\rangle _{H}
\]
 for all $x\in D\left(A\right)$. Thus, we find $P_{n}y\in D\left(A^{*}\right)$
and 
\[
A^{*}P_{n}y=P_{n}A^{*}y.
\]
This shows that $P_{n}$ commutes with $A^{*}$ and similarly we find
$P_{n}$ commuting with $B^{*}$. 

From $P_{n}Bx=BP_{n}x$ for $x\in D\left(B\right)$ we get
\begin{align*}
\left\langle x|\left(P_{n}BP_{n}\right)^{*}y\right\rangle _{H} & =\left\langle P_{n}BP_{n}x|y\right\rangle _{H}=\left\langle BP_{n}x|P_{n}y\right\rangle _{H}=\left\langle Bx|P_{n}y\right\rangle _{H}
\end{align*}
 for all $x\in D\left(B\right)$, $y\in H$ and so
\[
\left\langle x|\left(P_{n}BP_{n}\right)^{*}y\right\rangle _{H}=\left\langle P_{n}x|B^{*}P_{n}y\right\rangle _{H}=\left\langle x|P_{n}B^{*}P_{n}y\right\rangle _{H}
\]
proving 
\[
\left(P_{n}BP_{n}\right)^{*}=P_{n}B^{*}P_{n}
\]
for every $n\in\mathbb{N}$.

Let now $y\in D\left(A^{^{*}}+B^{^{*}}\right)$ then $\left\langle \left(A+B\right)x|y\right\rangle _{H}=\left\langle x|\left(A+B\right)^{*}y\right\rangle _{H}$
for all $x\in D\left(A+B\right)$. Since $P_{n}\left(A+B\right)\subseteq\left(A+B\right)P_{n}$
we have
\begin{align*}
\left\langle x|P_{n}\left(A+B\right)^{*}y\right\rangle _{H} & =\left\langle P_{n}x|\left(A+B\right)^{*}y\right\rangle _{H}\\
 & =\left\langle \left(A+B\right)P_{n}x|y\right\rangle _{H}\\
 & =\left\langle P_{n}\left(A+B\right)P_{n}x|y\right\rangle _{H}\\
 & =\left\langle P_{n}AP_{n}x+P_{n}BP_{n}x|y\right\rangle _{H}\\
 & =\left\langle P_{n}AP_{n}x|y\right\rangle _{H}+\left\langle x|P_{n}B^{*}P_{n}y\right\rangle _{H}
\end{align*}
implying $y\in D\left(\left(P_{n}AP_{n}\right)^{*}\right)$ and $\left(P_{n}AP_{n}\right)^{*}y=P_{n}\left(A+B\right)^{*}y-P_{n}B^{*}P_{n}y$.
Thus, we see that 
\begin{align*}
P_{n}\left(A+B\right)^{*}\subseteq & \left(P_{n}AP_{n}\right)^{*}+P_{n}B^{*}P_{n}
\end{align*}
and so, since clearly we have
\[
A^{*}+B^{*}\subseteq\left(A+B\right)^{*},
\]
we have the relations
\begin{equation}
P_{n}A^{*}P_{n}+P_{n}B^{*}P_{n}=P_{n}\left(A^{*}+B^{*}\right)P_{n}\subseteq P_{n}\left(A+B\right)^{*}P_{n}\subseteq\left(P_{n}AP_{n}\right)^{*}+P_{n}B^{*}P_{n}.\label{eq:adjadj-1}
\end{equation}
In particular, this yields
\[
P_{n}A^{*}P_{n}\subseteq\left(P_{n}AP_{n}\right)^{*}
\]
for every $n\in\mathbb{N}$. 

Let now $y\in D\left(\left(P_{n}AP_{n}\right)^{*}\right)$ then
\[
\left\langle P_{n}AP_{n}x|y\right\rangle _{H}=\left\langle x|\left(P_{n}AP_{n}\right)^{*}y\right\rangle _{H}
\]
for all $x\in D\left(P_{n}AP_{n}\right).$ In particular, for $x\in D\left(A\right)$
we get
\[
\left\langle Ax|P_{n}y\right\rangle _{H}=\left\langle P_{n}Ax|y\right\rangle _{H}=\left\langle P_{n}AP_{n}x|y\right\rangle _{H}=\left\langle x|\left(P_{n}AP_{n}\right)^{*}y\right\rangle _{H}
\]
showing that $P_{n}y\in D\left(A^{*}\right)$ and 
\[
A^{*}P_{n}y=\left(P_{n}AP_{n}\right)^{*}y.
\]
Moreover,
\begin{align*}
\left\langle x|\left(P_{n}AP_{n}\right)^{*}y\right\rangle _{H} & =\left\langle P_{n}AP_{n}x|y\right\rangle _{H}=\left\langle P_{n}x|A^{*}P_{n}y\right\rangle _{H}\\
 & =\left\langle x|P_{n}A^{*}P_{n}y\right\rangle _{H}
\end{align*}
for all $x\in D\left(A\right)$ proving that
\[
\left(P_{n}AP_{n}\right)^{*}\subseteq P_{n}A^{*}P_{n}.
\]
 Thus, we also have 
\[
\left(P_{n}AP_{n}\right)^{*}=P_{n}A^{*}P_{n}.
\]
 From (\ref{eq:adjadj-1}) we obtain now
\begin{eqnarray*}
P_{n}\left(A+B\right)^{*}\subseteq P_{n}A^{*}P_{n}+P_{n}B^{*}P_{n} & = & P_{n}\left(A^{*}+B^{*}\right)P_{n}\\
 &  & =P_{n}\left(A+B\right)^{*}P_{n}=\left(P_{n}AP_{n}\right)^{*}+P_{n}B^{*}P_{n}.
\end{eqnarray*}
I.e. for $z\in\left(A+B\right)^{*}$ we have
\begin{align*}
P_{n}\left(A+B\right)^{*}z & =P_{n}\left(A^{*}+B^{*}\right)P_{n}z\,,\\
 & =\left(A^{*}+B^{*}\right)P_{n}z.
\end{align*}
Since $P_{n}\conv{n}1$ and from the closability of $A^{*}+B^{*}$
follows $z\in D\left(\overline{A^{*}+B^{*}}\right)$ and 
\[
\overline{A^{*}+B^{*}}z=\left(A+B\right)^{*}z.
\]
Thus we have indeed shown that 
\[
\left(A+B\right)^{*}=\overline{A^{*}+B^{*}}.
\]

\end{proof}
This lemma will be crucial in the proof of our solution theorem.
\begin{thm}
(Solution Theory) Let $A$ and $M$ be as above and satisfy \textbf{Condition
(Positivity \ref{Condition-(Positivity I)})}. Then for every $f\in H_{\rho}\left(\mathbb{R},H\right)$,
$\rho\in\mathbb{R}_{>0}$ sufficiently large, there is a unique solution
$U\in H_{\rho}\left(\mathbb{R},H\right)$ of 
\[
\left(\partial_{0}M\left(\partial_{0}^{-1}\right)+A\right)U=f.
\]
The solution depends continuously on the data in the sense that
\[
\left|U\right|_{\rho,0,0}\leq\beta_{0}^{-1}\left|f\right|_{\rho,0,0}
\]
uniformly for all $f\in H_{\rho}\left(\mathbb{R},H\right)$ and $\rho\in\mathbb{R}_{>0}$
sufficiently large. \end{thm}
\begin{proof}
From \textbf{Condition (Positivity \ref{Condition-(Positivity III)})}
we see that $\left(\partial_{0}M\left(\partial_{0}^{-1}\right)+A\right)$
and  $\left(\partial_{0}^{*}M^{*}\left(\left(\partial_{0}^{-1}\right)^{*}\right)+A^{*}\right)$
have both inverses bounded by $\beta_{0}^{-1}$. The invertibility
of $\left(\partial_{0}M\left(\partial_{0}^{-1}\right)+A\right)$ already
confirms the continuous dependence estimate. Moreover, we know that
the null spaces of these operators are trivial. In particular,
\begin{equation}
N\left(\overline{\partial_{0}^{*}M^{*}\left(\left(\partial_{0}^{-1}\right)^{*}\right)+A^{*}}\right)=\left\{ 0\right\} .\label{eq:nullstar}
\end{equation}
It remains to be seen that the range $\left(\partial_{0}M\left(\partial_{0}^{-1}\right)+A\right)\left[H_{\rho}\left(\mathbb{R},H\right)\right]$
of $\left(\partial_{0}M\left(\partial_{0}^{-1}\right)+A\right)$ is
dense in $H_{\rho}\left(\mathbb{R},H\right)$, then the result follows
(recall that $\left(\partial_{0}M\left(\partial_{0}^{-1}\right)+A\right)$
is used in the above as a suggestive notation for $\overline{\partial_{0}M\left(\partial_{0}^{-1}\right)+A}$).

The previous lemma applied with%
\footnote{Recall that $\left(\chi_{_{\left]-\infty,\lambda\right]}}\left(\Im\left(\partial_{0}\right)\right)\right)_{\lambda\in\mathbb{R}}$
is the spectral family associated with the selfadjoint operator $\Im\left(\partial_{0}\right)$.%
} \vspace{-0.3cm}
\[
P_{n}\::=\chi_{_{\left[-n,n\right]}}\left(\Im\left(\partial_{0}\right)\right),\: n\in\mathbb{N},
\]
and $B\::=\partial_{0}M\left(\partial_{0}^{-1}\right)$ yields
\[
\left(\partial_{0}M\left(\partial_{0}^{-1}\right)+A\right)^{*}=\overline{\partial_{0}^{*}M^{*}\left(\left(\partial_{0}^{-1}\right)^{*}\right)+A^{*}}.
\]
Since from the projection theorem we have the orthogonal decomposition
\[
H_{\rho}\left(\mathbb{R},H\right)=N\left(\left(\partial_{0}M\left(\partial_{0}^{-1}\right)+A\right)^{*}\right)\oplus\overline{\left(\partial_{0}M\left(\partial_{0}^{-1}\right)+A\right)\left[H_{\rho}\left(\mathbb{R},H\right)\right]}
\]
and from (\ref{eq:nullstar}) we see that
\[
N\left(\left(\partial_{0}M\left(\partial_{0}^{-1}\right)+A\right)^{*}\right)=\left\{ 0\right\} ,
\]
it follows indeed that 
\[
H_{\rho}\left(\mathbb{R},H\right)=\overline{\left(\partial_{0}M\left(\partial_{0}^{-1}\right)+A\right)\left[H_{\rho}\left(\mathbb{R},H\right)\right]}.
\]

\end{proof}
This well-posedness result is, however, not all we would like to have.
Note that so far we have only used \textbf{Condition (Positivity \ref{Condition-(Positivity III)})}
which was a simple consequence of \textbf{Condition (Positivity \ref{Condition-(Positivity I)})}.
For an actual \emph{evolution} to take place we also need additionally
a property securing causality for the solution. This is where the
\textbf{Condition (Positivity \ref{Condition-(Positivity I)}) }comes
into play.

\section{\label{sec:Causality-and-Memory}Causality }

We first need to specify what we mean by causality. This can here
be done here in a more elementary way than in \cite{Pi2009-1}. 
\begin{defn}
\label{def:causality}A mapping $F:H_{\rho}\left(\mathbb{R},H\right)\to H_{\rho}\left(\mathbb{R},H\right)$
is called causal if
\[
\bigwedge_{a\in\mathbb{R}}\;\bigwedge_{u,v\in H_{\rho}\left(\mathbb{R},H\right)}\left(\chi_{_{\left]-\infty,a\right]}}\left(m_{0}\right)\:\left(u-v\right)=0\;\implies\chi_{_{\left]-\infty,a\right]}}\left(m_{0}\right)\:\left(F\left(u\right)-F\left(v\right)\right)=0\right).
\]
\end{defn}
\begin{rem}
Note that if $F$ is translation invariant, then $a\in\mathbb{R}$
in this statement can be fixed (for example to $a=0$). If $F$ is
linear we may fix $v=0$ to simplify the requirement.
\end{rem}
It is known that, by construction from an analytic, bounded $M$,
the operator $M\left(\partial_{0}^{-1}\right)$ is causal, see \cite{2009-2}. 
\begin{thm}
\label{thm:Causality-of-Solution}(Causality of Solution Operator)
Let $A$ and $M$ be as above and satisfy \textbf{Condition (Positivity
\ref{Condition-(Positivity I)})}. Then the solution operator
\[
\left(\partial_{0}M\left(\partial_{0}^{-1}\right)+A\right)^{-1}:H_{\rho}\left(\mathbb{R},H\right)\to H_{\rho}\left(\mathbb{R},H\right)
\]
is causal for all sufficiently large $\rho\in\mathbb{R}_{>0}.$ \end{thm}
\begin{proof}
By translation invariance we may base our arguments on \textbf{Condition
(Positivity \ref{Condition-(Positivity II)})}. We estimate 
\begin{eqnarray*}
 &  & \left|\chi_{_{\mathbb{R}\leq a}}\left(m_{0}\right)U\right|_{\rho,0,0}\left|\chi_{_{\mathbb{R}\leq a}}\left(m_{0}\right)\left(\partial_{0}M\left(\partial_{0}^{-1}\right)+A\right)U\right|_{\rho,0,0}\geq\\
 &  & \geq\Re\left\langle \chi_{_{\mathbb{R}\leq a}}\left(m_{0}\right)U|\left(\partial_{0}M\left(\partial_{0}^{-1}\right)+A\right)U\right\rangle _{\rho,0,0},\\
 &  & \geq\beta_{0}\left|\chi_{_{\mathbb{R}\leq a}}\left(m_{0}\right)U\right|_{\rho,0,0}^{2},
\end{eqnarray*}
yielding
\[
\beta_{0}\left|\chi_{_{\mathbb{R}\leq a}}\left(m_{0}\right)U\right|_{\rho,0,0}\leq\left|\chi_{_{\mathbb{R}\leq a}}\left(m_{0}\right)\left(\partial_{0}M\left(\partial_{0}^{-1}\right)+A\right)U\right|_{\rho,0,0}
\]
for every $a\in\mathbb{R}$. Substituting $\left(\partial_{0}M\left(\partial_{0}^{-1}\right)+A\right)^{-1}f$
for $U$ this gives
\begin{equation}
\beta_{0}\left|\chi_{_{\mathbb{R}\leq a}}\left(m_{0}\right)\left(\partial_{0}M\left(\partial_{0}^{-1}\right)+A\right)^{-1}f\right|_{\rho,0,0}\leq\left|\chi_{_{\mathbb{R}\leq a}}\left(m_{0}\right)f\right|_{\rho,0,0},\: a\in\mathbb{R}.\label{eq:causal:estimate}
\end{equation}
We read off that if $\chi_{_{\mathbb{R}\leq a}}\left(m_{0}\right)f=0$
then $\chi_{_{\mathbb{R}\leq a}}\left(m_{0}\right)\left(\partial_{0}M\left(\partial_{0}^{-1}\right)+A\right)^{-1}f=0$,
$a\in\mathbb{R}$. This is the desired causality of $\left(\partial_{0}M\left(\partial_{0}^{-1}\right)+A\right)^{-1}$.
\end{proof}
\selectlanguage{english}%
\section[Acoustic Waves with Impedance Type Boundary Condition]{An Application: Acoustic Waves with Impedance Type Boundary Condition
}

\selectlanguage{american}%
We want to conclude our discussion with a more substantial utilization
of the theory presented. We assume that the material law is of the
form 
\[
M\left(\partial_{0}^{-1}\right)=M_{0}+\partial_{0}^{-1}M_{1}\left(\partial_{0}^{-1}\right)
\]
with $M_{0}$ selfadjoint and strictly positive definite. The underlying
Hilbert space $H=L^{2}\left(\Omega\right)\oplus L^{2}\left(\Omega\right)^{3}$
and so we consider the material law as an operator in the spaces\linebreak{}
 $H_{\rho}\left(\mathbb{R},L^{2}\left(\Omega\right)\oplus L^{2}\left(\Omega\right)^{3}\right)$,
$\rho\in\mathbb{R}_{>0}$. 

The Hilbert space $H\left(\interior{\dive},\Omega\right)$ is (equipped
with the graph norm) the domain of the closure of the classical divergence
operator on vector fields with $\interior C_{\infty}\left(\Omega,\mathbb{C}\right)$
considered as a mapping from $L^{2}\left(\Omega\right)^{3}$ to $L^{2}\left(\Omega\right)$.
To be an element of $H\left(\interior{\dive},\Omega\right)$ generalizes
the classical boundary condition of vanishing normal component on
the boundary of $\Omega$ to cases with non-smooth boundary. Analogously
we define the Hilbert space $H\left(\interior{\grad},\Omega\right)$
as the domain of the closure of the classical gradient operator on
functions in $\interior C_{\infty}\left(\Omega,\mathbb{C}\right)$
(equipped with the graph norm). 

Moreover, we let 
\[
A=\left(\begin{array}{cc}
0 & \dive\\
\grad & 0
\end{array}\right)
\]
with 
\[
D\left(A\right):=\left\{ \left(\begin{array}{c}
p\\
v
\end{array}\right)\in D\left(\left(\begin{array}{cc}
0 & \dive\\
\grad & 0
\end{array}\right)\right)\Big|\: a\left(\partial_{0}^{-1}\right)p-\partial_{0}^{-1}v\in H_{\rho}\left(\mathbb{R},H\left(\interior{\dive},\Omega\right)\right)\right\} .
\]
Here we assume that $a=\left(a\left(z\right)\right)_{z\in B_{\mathbb{C}}\left(r,r\right)}$
is analytic and bounded and that $a\left(\partial_{0}^{-1}\right)$
is multiplicative in the sense that it is of the form 
\[
a\left(z\right)=\sum_{k=0}^{\infty}a_{k,r}\left(m\right)\,\left(z-r\right)^{k},
\]
where $a_{k,r}$ are $L^{\infty}\left(\Omega\right)$-vector fields
and $a_{k,r}\left(m\right)$ is the associated multiplication operator
\[
\left(a_{k,r}\left(m\right)\varphi\right)\left(t,x\right)\::=a_{k,r}\left(x\right)\,\varphi\left(t,x\right)
\]
for $\varphi\in\interior C_{\infty}\left(\mathbb{R}\times\Omega,\mathbb{C}\right)$.
Moreover, we assume that 
\[
\left(\dive a\right)\left(z\right)\;:=\sum_{k=0}^{\infty}\left(\dive a_{k,r}\right)\left(m\right)\,\left(z-r\right)^{k}
\]
is analytic and bounded with $\dive a_{k,r}\in L^{\infty}\left(\Omega\right)$.
Then, in particular, the product rule holds
\begin{equation}
\dive\left(a\left(\partial_{0}^{-1}\right)p\right)=\left(\dive a\left(\partial_{0}^{-1}\right)\right)p+a\left(\partial_{0}^{-1}\right)\cdot\grad p\,.\label{eq:productrule}
\end{equation}
As a consequence, we have
\begin{equation}
a\left(\partial_{0}^{-1}\right):H_{\rho}\left(\mathbb{R},H\left(\grad,\Omega\right)\right)\to H_{\rho}\left(\mathbb{R},H\left(\dive,\Omega\right)\right)\label{eq:divbound}
\end{equation}
uniformly bounded for all sufficiently large $\rho\in\mathbb{R}_{>0}$. 

The boundary condition given in the description of the domain of $A$
is in classical terms%
\footnote{This includes as highly special cases boundary conditions of the form
$kp\left(t\right)-n\cdot v\left(t\right)=0$ (Robin boundary condition),
$k\partial_{0}p\left(t\right)-n\cdot v\left(t\right)=0$ or $kp\left(t\right)-n\cdot\partial_{0}v\left(t\right)=0$,
$\mbox{ on }\partial\Omega,\: t\in\mathbb{R},$ $k\in\mathbb{R}_{>0}$.
It should be noted that in the time-independent case the above sign
constraints become void since causality is not an issue anymore and
in simple cases the problem is elliptic, which can be dealt with by
sesqui-linear form methods, compare e.g. \cite[section 2.4]{Leis:Buch:2}. 

The general class of boundary conditions considered here in the time-dependent,
time-translation invariant case covers for example cases of additional
temporal convolution terms also on the boundary. %
} 
\[
n\cdot a\left(\partial_{0}^{-1}\right)\partial_{0}p\left(t\right)-n\cdot v\left(t\right)\:=0\mbox{ on }\partial\Omega,\: t\in\mathbb{R},
\]
in the case that the boundary $\partial\Omega$ of $\Omega$ and the
solution are smooth and $\partial\Omega$ has $n$ as exterior unit
normal field.

We also will impose a sign requirement on $a$:
\begin{equation}
\Re\int_{-\infty}^{0}\left(\left\langle \grad p|\partial_{0}a\left(\partial_{0}^{-1}\right)p\right\rangle _{0}\left(t\right)\:+\left\langle p|\dive\partial_{0}a\left(\partial_{0}^{-1}\right)p\right\rangle _{0}\left(t\right)\right)\:\exp\left(-2\rho t\right)\: dt\geq0\label{eq:apos}
\end{equation}
for all sufficiently large $\rho\in\mathbb{R}_{>0}$. This is the
appropriate generalization of the condition: 
\[
\Re\int_{\partial\Omega}p\left(t\right)^{*}\left(\partial_{0}n\cdot a\left(\partial_{0}^{-1}\right)p\right)\left(t\right)\: do\geq0,\: t\in\mathbb{R},
\]
in the case that $\partial\Omega$ and $p$ are smooth and $n$ is
the exterior unit normal field. 
\begin{rem}
We could require instead that the quadratic functional $Q_{\Omega,a\left(\partial_{0}^{-1}\right)}$
given by
\begin{eqnarray*}
p\mapsto\left\langle \grad p|\Re\left(\partial_{0}a\left(\partial_{0}^{-1}\right)\right)p\right\rangle _{\rho,0,0}+\left\langle \Re\left(\partial_{0}a\left(\partial_{0}^{-1}\right)\right)p|\grad p\right\rangle _{\rho,0,0}+\\
+\left\langle \dive\left(\Re\left(\partial_{0}a\left(\partial_{0}^{-1}\right)\right)\right)p|p\right\rangle _{\rho,0,0}
\end{eqnarray*}
is non-negative on $H\left(\grad,\,\Omega\right)$. 

Note that this functional vanishes on $H\left(\interior{\grad},\Omega\right)$
and therefore the positivity condition constitutes a boundary constraint%
\footnote{A boundary constraint is a proposition $P$ on a mathematical object
expressed in terms of its interaction with a function space over a
domain $\Omega\subseteq\mathbb{R}^{n+1}$, $n\in\mathbb{N}$, which
is true for all subspaces of elements with compact support. A boundary
condition is a proposition imposed on elements $u$ of a function
space over a domain $\Omega\subseteq\mathbb{R}^{n+1}$, $n\in\mathbb{N}$,
which is also satisfied for $u+\varphi$ for all $\varphi$ in the
function space having compact support in $\Omega$. An example for
a boundary constraint for an open set $\Omega\subseteq\mathbb{R}^{n+1}$,
$n\in\mathbb{N}$, is: $H\left(\interior{\grad},\Omega\right)$ is
compactly embedded into $L^{2}\left(\Omega\right)$. This constraint
is non-trivial, since there are cases in which this is not true. On
the other hand, the property is true for every subspace $\left\{ \varphi\in H\left(\interior{\grad},\Omega\right)|\:\mathrm{supp}\varphi\subseteq K\subset\subset\Omega\right\} $. 

Given $f\in H\left({\grad},\Omega\right)$, imposing on $u\in H\left({\grad},\Omega\right)$
the requirement $u-f\in H\left(\interior{\grad},\Omega\right)$ is
a boundary condition. Indeed, if $H\left(\interior{\grad},\Omega\right)\not=H\left({\grad},\Omega\right)$
the proposition $u-f\in H\left(\interior{\grad},\Omega\right)$ is
non-trivial, since there are elements $u\in H\left({\grad},\Omega\right)$
not satisfying the proposition, and obviously $u+\varphi-f\in H\left(\interior{\grad},\Omega\right)$
for every $\varphi\in H\left({\grad},\Omega\right)$ with compact
support in $\Omega$, since such $\varphi$ is in $H\left(\interior{\grad},\Omega\right)$. %
}on $a\left(\partial_{0}^{-1}\right)$ and on the underlying domain
$\Omega$. The constraint on $\Omega$ is that the requirement $Q_{\Omega,a\left(\partial_{0}^{-1}\right)}\left[H\left({\grad},\Omega\right)\right]\subseteq\mathbb{R}_{\geq0}$
must be non-trivial, i.e. there must be an $a\left(\partial_{0}^{-1}\right)$
for which this does not hold. For this surely we must have $H\left(\interior{\grad},\Omega\right)\not=H\left({\grad},\Omega\right)$.\end{rem}
\begin{prop}
\label{pro:Apos}Let $A$ be as given above. Then $A$ is closed,
densely defined and 
\begin{eqnarray*}
\Re\left\langle \chi_{_{\mathbb{R}_{<0}}}\left(m_{0}\right)\, U|AU\right\rangle _{\rho,0,0} & \geq & 0
\end{eqnarray*}
for all sufficiently large $\rho\in\mathbb{R}_{>0}$ and all $U\in D\left(A\right)$.\end{prop}
\begin{proof}
Any $U$ with components in $\interior C_{\infty}\left(\mathbb{R}\times\Omega\right)$
is in $D\left(A\right)$. Note that 
\[
U\in D\left(A\right)
\]
is equivalent to
\[
\left(\begin{array}{cc}
1 & 0\\
-a\left(\partial_{0}^{-1}\right) & \partial_{0}^{-1}
\end{array}\right)U\in H_{\rho}\left(\mathbb{R},H\left({\grad},\Omega\right)\oplus H\left(\interior{\dive},\Omega\right)\right).
\]
According to (\ref{eq:divbound}) we have that{\small{} 
\begin{equation}
\left(\begin{array}{cc}
1 & 0\\
-a\left(\partial_{0}^{-1}\right) & \partial_{0}^{-1}
\end{array}\right):H_{\rho}\left(\mathbb{R},H\left({\grad},\Omega\right)\oplus H\left(\dive,\Omega\right)\right)\to H_{\rho}\left(\mathbb{R},H\left({\grad},\Omega\right)\oplus H\left(\dive,\Omega\right)\right)\label{eq:contbound}
\end{equation}
}is a well-defined continuous linear mapping. Moreover, since multiplication
by an $L^{\infty}\left(\Omega\right)$-multiplier and application
by $\partial_{0}^{-1}$ does not increase the support, we have that
if $\Phi$ has support in $\mathbb{R}\times K$ for some compact set
$K\subseteq\Omega$ then $\left(\begin{array}{cc}
1 & 0\\
-a\left(\partial_{0}^{-1}\right) & \partial_{0}^{-1}
\end{array}\right)\Phi$ also has support in $\mathbb{R}\times K$. This confirms that $\interior C_{\infty}\left(\mathbb{R}\times\Omega\right)\subseteq D\left(A\right)$
and since $\interior C_{\infty}\left(\mathbb{R}\times\Omega\right)$
is dense in $H_{\rho}\left(\mathbb{R},L^{2}\left(\Omega\right)\oplus L^{2}\left(\Omega\right)\right)$.

Let now $\Phi_{k}\conv{k}\Phi_{\infty}$ and $A\Phi_{k}\conv{k}\Psi_{\infty}$
we have first, due to the closedness of $\left(\begin{array}{cc}
0 & \dive\\
\grad & 0
\end{array}\right)$ that $\Psi_{\infty}=\left(\begin{array}{cc}
0 & \dive\\
\grad & 0
\end{array}\right)\Phi_{\infty}$. Moreover, we have from (\ref{eq:contbound})
\[
\left(\begin{array}{cc}
1 & 0\\
-a\left(\partial_{0}^{-1}\right) & \partial_{0}^{-1}
\end{array}\right)\Phi_{k}\conv{k}\left(\begin{array}{cc}
1 & 0\\
-a\left(\partial_{0}^{-1}\right) & \partial_{0}^{-1}
\end{array}\right)\Phi_{\infty}.
\]
Using (\ref{eq:productrule}), a straightforward calculation yields
on $D\left(A\right)$
\begin{eqnarray*}
\left(\begin{array}{cc}
0 & \interior{\dive}\\
\grad & 0
\end{array}\right)\left(\begin{array}{cc}
1 & 0\\
-a\left(\partial_{0}^{-1}\right) & \partial_{0}^{-1}
\end{array}\right) & = & \left(\begin{array}{cc}
0 & \dive\\
\grad & 0
\end{array}\right)\left(\begin{array}{cc}
1 & 0\\
-a\left(\partial_{0}^{-1}\right) & \partial_{0}^{-1}
\end{array}\right)\\
 & = & \left(\begin{array}{cc}
\partial_{0}^{-1} & 0\\
0 & 1
\end{array}\right)\left(\begin{array}{cc}
0 & \dive\\
\grad & 0
\end{array}\right)+\left(\begin{array}{cc}
0 & \dive\\
0 & 0
\end{array}\right)\left(\begin{array}{cc}
0 & 0\\
-a\left(\partial_{0}^{-1}\right) & 0
\end{array}\right)\\
 & = & \left(\begin{array}{cc}
\partial_{0}^{-1} & 0\\
0 & 1
\end{array}\right)\left(\begin{array}{cc}
0 & \dive\\
\grad & 0
\end{array}\right)+\left(\begin{array}{cc}
-\dive a\left(\partial_{0}^{-1}\right) & 0\\
0 & 0
\end{array}\right)\\
 & = & \left(\begin{array}{cc}
0 & \dive\\
\grad & 0
\end{array}\right)-\left(\begin{array}{cc}
\left(\dive a\right)\left(\partial_{0}^{-1}\right) & 0\\
0 & 0
\end{array}\right)-\left(\begin{array}{cc}
a\left(\partial_{0}^{-1}\right)\cdot\grad & 0\\
0 & 0
\end{array}\right)\\
 & = & \left(\begin{array}{cc}
\partial_{0}^{-1} & -a\left(\partial_{0}^{-1}\right)\cdot\\
0 & 1
\end{array}\right)\left(\begin{array}{cc}
0 & \dive\\
\grad & 0
\end{array}\right)-\left(\begin{array}{cc}
\left(\dive a\right)\left(\partial_{0}^{-1}\right) & 0\\
0 & 0
\end{array}\right).
\end{eqnarray*}
Thus, we have
\begin{equation}
\begin{array}{rcl}
\partial_{0}^{-1}\left(\begin{array}{cc}
0 & \dive\\
\grad & 0
\end{array}\right)U & = & \left(\begin{array}{cc}
1 & a\left(\partial_{0}^{-1}\right)\cdot\\
0 & \partial_{0}^{-1}
\end{array}\right)\left(\begin{array}{cc}
0 & \interior{\dive}\\
\grad & 0
\end{array}\right)\left(\begin{array}{cc}
1 & 0\\
-a\left(\partial_{0}^{-1}\right) & \partial_{0}^{-1}
\end{array}\right)U+\\
 &  & +\left(\begin{array}{cc}
\left(\dive a\right)\left(\partial_{0}^{-1}\right) & 0\\
0 & 0
\end{array}\right)U.
\end{array}\label{eq:Arep0}
\end{equation}
 Here we have used that 
\[
\left(\begin{array}{cc}
\partial_{0}^{-1} & -a\left(\partial_{0}^{-1}\right)\cdot\\
0 & 1
\end{array}\right)^{-1}=\partial_{0}\left(\begin{array}{cc}
1 & a\left(\partial_{0}^{-1}\right)\cdot\\
0 & \partial_{0}^{-1}
\end{array}\right).
\]
Consequently, 
\begin{eqnarray*}
 &  & \left(\begin{array}{cc}
0 & \interior{\dive}\\
\grad & 0
\end{array}\right)\left(\begin{array}{cc}
1 & 0\\
-a\left(\partial_{0}^{-1}\right) & \partial_{0}^{-1}
\end{array}\right)\Phi_{k}=\\
 &  & =\left(\begin{array}{cc}
\partial_{0}^{-1} & -a\left(\partial_{0}^{-1}\right)\\
0 & 1
\end{array}\right)\left(\begin{array}{cc}
0 & \dive\\
\grad & 0
\end{array}\right)\Phi_{k}-\left(\begin{array}{cc}
\left(\dive a\right)\left(\partial_{0}^{-1}\right) & 0\\
0 & 0
\end{array}\right)\Phi_{k}\\
 &  & \conv{k}\left(\begin{array}{cc}
\partial_{0}^{-1} & -a\left(\partial_{0}^{-1}\right)\\
0 & 1
\end{array}\right)\Psi_{\infty}-\left(\begin{array}{cc}
\left(\dive a\right)\left(\partial_{0}^{-1}\right) & 0\\
0 & 0
\end{array}\right)\Phi_{\infty}
\end{eqnarray*}
and so, by the closedness of $\left(\begin{array}{cc}
0 & \interior{\dive}\\
\grad & 0
\end{array}\right)$
\[
\left(\begin{array}{cc}
1 & 0\\
-a\left(\partial_{0}^{-1}\right) & \partial_{0}^{-1}
\end{array}\right)\Phi_{\infty}\in H_{\rho}\left(\mathbb{R},H\left({\grad},\Omega\right)\oplus H\left(\interior{\dive},\Omega\right)\right)
\]
and so
\[
\Phi_{\infty}\in D\left(A\right).
\]
Moreover, we have from (\ref{eq:Arep0}) 
\begin{equation}
\begin{array}{rcl}
\partial_{0}^{-1}A\Phi_{\infty} & = & \left(\begin{array}{cc}
1 & a\left(\partial_{0}^{-1}\right)\cdot\\
0 & \partial_{0}^{-1}
\end{array}\right)\left(\begin{array}{cc}
0 & \interior{\dive}\\
\grad & 0
\end{array}\right)\left(\begin{array}{cc}
1 & 0\\
-a\left(\partial_{0}^{-1}\right) & \partial_{0}^{-1}
\end{array}\right)\Phi_{\infty}+\\
 &  & +\left(\begin{array}{cc}
\left(\dive a\right)\left(\partial_{0}^{-1}\right) & 0\\
0 & 0
\end{array}\right)\Phi_{\infty}.
\end{array}\label{eq:Arep1}
\end{equation}
 We shall now show the (real) non-negativity of $A$. Assume that
\[
U=\left(\begin{array}{c}
p\\
v
\end{array}\right)\in\bigcup_{n\in\mathbb{N}}\chi_{_{\left[-n,n\right]}}\left(\Im\left(\partial_{0}\right)\right)\left[D\left(A\right)\right]
\]
 then $U\in D\left(\partial_{0}\right)\cap D\left(A\right)$ and we
may calculate

\begin{eqnarray*}
 &  & \Re\left\langle U|AU\right\rangle _{0}=\\
 &  & =\frac{1}{2}\left(\left\langle U|AU\right\rangle _{0}+\left\langle AU|U\right\rangle _{0}\right)\\
 &  & =\frac{1}{2}\left(\left(\left\langle p|\dive v\right\rangle _{0}+\left\langle \grad p|v\right\rangle _{0}\right)+\left(\left\langle \dive v|p\right\rangle _{0}+\left\langle v|\grad p\right\rangle _{0}\right)\right)\\
 &  & =\frac{1}{2}\left(\left\langle p|\interior{\dive}\left(v-\partial_{0}a\left(\partial_{0}^{-1}\right)p\right)\right\rangle _{0}+\left\langle p|\dive\partial_{0}a\left(\partial_{0}^{-1}\right)p\right\rangle _{0}+\left\langle \grad p|v\right\rangle _{\rho,0}\right)+\\
 &  & +\frac{1}{2}\left(\left\langle \interior{\dive}\left(v-\partial_{0}a\left(\partial_{0}^{-1}\right)p\right)|p\right\rangle _{0}+\left\langle \dive\partial_{0}a\left(\partial_{0}^{-1}\right)p|p\right\rangle _{0}+\left\langle v|\grad p\right\rangle _{0}\right)\\
 &  & =\frac{1}{2}\left(-\left\langle \grad p|\left(v-\partial_{0}a\left(\partial_{0}^{-1}\right)p\right)\right\rangle _{0}+\left\langle p|\dive\partial_{0}a\left(\partial_{0}^{-1}\right)p\right\rangle _{0}+\left\langle \grad p|v\right\rangle _{0}\right)+\\
 &  & +\frac{1}{2}\left(-\left\langle \left(v+\partial_{0}a\left(\partial_{0}^{-1}\right)p\right)|\grad p\right\rangle _{0}+\left\langle \dive\partial_{0}a\left(\partial_{0}^{-1}\right)p|p\right\rangle _{0}+\left\langle v|\grad p\right\rangle _{0}\right)\\
 &  & =\frac{1}{2}\left(\left\langle \grad p|\partial_{0}a\left(\partial_{0}^{-1}\right)p\right\rangle _{0}+\left\langle p|\dive\partial_{0}a\left(\partial_{0}^{-1}\right)p\right\rangle _{0}\right)+\\
 &  & +\frac{1}{2}\left(\left\langle \partial_{0}a\left(\partial_{0}^{-1}\right)p|\grad p\right\rangle _{0}+\left\langle \dive\partial_{0}a\left(\partial_{0}^{-1}\right)p|p\right\rangle _{0}\right).
\end{eqnarray*}
Integrating this over $\mathbb{R}_{<0}$ yields with requirement (\ref{eq:apos})
\[
\Re\left\langle \chi_{_{\mathbb{R}_{<0}}}\left(m_{0}\right)\, U|AU\right\rangle _{\rho,0,0}\geq0
\]
for $U\in\bigcup_{n\in\mathbb{N}}\chi_{_{\left[-n,n\right]}}\left(\Im\left(\partial_{0}\right)\right)\left[D\left(A\right)\right]$
and so by density for every $U\in D\left(A\right)$ (and all sufficiently
large $\rho\in\mathbb{R}_{>0}$). 
\end{proof}
We need to find the adjoint of $A$, which must be a restriction of 

\[
-\left(\begin{array}{cc}
0 & \dive\\
\grad & 0
\end{array}\right)
\]
and an extension of 
\[
-\left(\begin{array}{cc}
0 & \interior{\dive}\\
\interior{\grad} & 0
\end{array}\right).
\]
We suspect it is 
\[
D\left(A^{*}\right)\;:=\left\{ \left(\begin{array}{c}
p\\
v
\end{array}\right)\in D\left(\left(\begin{array}{cc}
0 & \dive\\
\grad & 0
\end{array}\right)\right)\Big|\: a\left(\partial_{0}^{-1}\right)^{*}p+\left(\partial_{0}^{-1}\right)^{*}v\in H_{\rho}\left(\mathbb{R},H\left(\interior{\dive},\Omega\right)\right)\right\} .
\]
Using (\ref{eq:Arep0}) and letting $W\;:=\partial_{0}\left(\begin{array}{cc}
1 & 0\\
-a\left(\partial_{0}^{-1}\right) & \partial_{0}^{-1}
\end{array}\right)U=\left(\begin{array}{cc}
\partial_{0}^{-1} & 0\\
a\left(\partial_{0}^{-1}\right) & 1
\end{array}\right)^{-1}U$ for 
\[
U\in\bigcup_{n\in\mathbb{N}}\chi_{_{\left[-n,n\right]}}\left(\Im\left(\partial_{0}\right)\right)\left[D\left(A\right)\right]
\]
we have $V\in D\left(A^{*}\right)$ if and only if for all $W\in\bigcup_{n\in\mathbb{N}}\chi_{_{\left[-n,n\right]}}\left(\Im\left(\partial_{0}\right)\right)\left[D\left(\left(\begin{array}{cc}
0 & \interior{\dive}\\
\grad & 0
\end{array}\right)\right)\right]$ 
\begin{eqnarray*}
0 & = & \left\langle \left(\begin{array}{cc}
1 & a\left(\partial_{0}^{-1}\right)\cdot\\
0 & \partial_{0}^{-1}
\end{array}\right)\left(\begin{array}{cc}
0 & \interior{\dive}\\
\grad & 0
\end{array}\right)W|V\right\rangle _{\rho,0,0}+\\
 &  & +\left\langle \left(\begin{array}{cc}
\left(\dive a\right)\left(\partial_{0}^{-1}\right) & 0\\
0 & 0
\end{array}\right)\partial_{0}\left(\begin{array}{cc}
\partial_{0}^{-1} & 0\\
a\left(\partial_{0}^{-1}\right) & 1
\end{array}\right)W|V\right\rangle _{\rho,0,0}+\\
 &  & +\left\langle \left(\begin{array}{cc}
\partial_{0}^{-1} & 0\\
a\left(\partial_{0}^{-1}\right) & 1
\end{array}\right)W|\left(\begin{array}{cc}
0 & \dive\\
\grad & 0
\end{array}\right)V\right\rangle _{\rho,0,0},\\
 & = & \left\langle \left(\begin{array}{cc}
1 & a\left(\partial_{0}^{-1}\right)\cdot\\
0 & \partial_{0}^{-1}
\end{array}\right)\left(\begin{array}{cc}
0 & \interior{\dive}\\
\grad & 0
\end{array}\right)W|V\right\rangle _{\rho,0,0}+\left\langle \left(\begin{array}{cc}
\left(\dive a\right)\left(\partial_{0}^{-1}\right) & 0\\
0 & 0
\end{array}\right)W|\, V\right\rangle _{\rho,0,0}+\\
 &  & +\left\langle \left(\begin{array}{cc}
\partial_{0}^{-1} & 0\\
a\left(\partial_{0}^{-1}\right) & 1
\end{array}\right)W|\left(\begin{array}{cc}
0 & \dive\\
\grad & 0
\end{array}\right)V\right\rangle _{\rho,0,0},\\
 & = & \left\langle \left(\begin{array}{cc}
0 & \interior{\dive}\\
\grad & 0
\end{array}\right)W|\left(\begin{array}{cc}
1 & 0\\
a\left(\partial_{0}^{-1}\right)^{*} & \left(\partial_{0}^{-1}\right)^{*}
\end{array}\right)V\right\rangle _{\rho,0,0}+\left\langle W|\,\left(\begin{array}{cc}
\dive\left(a\left(\partial_{0}^{-1}\right)^{*}\right) & 0\\
0 & 0
\end{array}\right)V\right\rangle _{\rho,0,0}+\\
 &  & +\left\langle W|\left(\begin{array}{cc}
\left(\partial_{0}^{-1}\right)^{*} & a\left(\partial_{0}^{-1}\right)^{*}\\
0 & 1
\end{array}\right)\left(\begin{array}{cc}
0 & \dive\\
\grad & 0
\end{array}\right)V\right\rangle _{\rho,0}.
\end{eqnarray*}

This implies that
\[
\left(\begin{array}{cc}
1 & 0\\
a\left(\partial_{0}^{-1}\right)^{*} & \left(\partial_{0}^{-1}\right)^{*}
\end{array}\right)V\in D\left(\left(\begin{array}{cc}
0 & \interior{\dive}\\
\grad & 0
\end{array}\right)\right),
\]
which is the above characterization. Moreover,
\[
\left(\begin{array}{cc}
0 & \interior{\dive}\\
\grad & 0
\end{array}\right)\left(\begin{array}{cc}
1 & 0\\
a\left(\partial_{0}^{-1}\right)^{*} & \left(\partial_{0}^{-1}\right)^{*}
\end{array}\right)V=\left(\begin{array}{cc}
\left(\dive a\right)\left(\partial_{0}^{-1}\right)^{*} & 0\\
0 & 0
\end{array}\right)V+\left(\begin{array}{cc}
\left(\partial_{0}^{-1}\right)^{*} & a\left(\partial_{0}^{-1}\right)^{*}\\
0 & 1
\end{array}\right)\left(\begin{array}{cc}
0 & \dive\\
\grad & 0
\end{array}\right)V
\]
which yields the analogous formula for $A^{*}$ as (\ref{eq:Arep0})
for $A$.
\begin{prop}
\label{pro:Asternpos}Let $A$ be as given above. Then $A^{*}$ is
closed, densely defined and 
\begin{eqnarray*}
\Re\left\langle \chi_{_{\mathbb{R}_{\leq0}}}\left(m_{0}\right)\, V|A^{*}V\right\rangle _{\rho,0,0} & \geq & 0
\end{eqnarray*}
for all sufficiently large $\rho\in\mathbb{R}_{>0}$ and all $V\in D\left(A^{*}\right)$.\end{prop}
\begin{proof}
The proof is analogous to the proof of Proposition \ref{pro:Apos},
since $A$ and $A^{*}$ share a similar structure. \end{proof}
\begin{thm}
Let $A$ and $M$ be as specified in this section, such that the propagation
of acoustic waves is governed by the equation
\[
\left(\partial_{0}M\left(\partial_{0}^{-1}\right)+A\right)U=f.
\]
Then, there is a $\rho_{0}\in\mathbb{R}_{>0}$ such that for every
$\rho\in\mathbb{R}_{\geq\rho_{0}}$ and $f\in H_{\rho}\left(\mathbb{R},L^{2}\left(\Omega\right)\oplus L^{2}\left(\Omega\right)^{3}\right)$
there is a unique solution $U\in H_{\rho}\left(\mathbb{R},L^{2}\left(\Omega\right)\oplus L^{2}\left(\Omega\right)^{3}\right)$.
Moreover, the solution operator $\left(\partial_{0}M\left(\partial_{0}^{-1}\right)+A\right)^{-1}:H_{\rho}\left(\mathbb{R},L^{2}\left(\Omega\right)\oplus L^{2}\left(\Omega\right)^{3}\right)\to H_{\rho}\left(\mathbb{R},L^{2}\left(\Omega\right)\oplus L^{2}\left(\Omega\right)^{3}\right)$
is a causal, continuous linear operator. Furthermore, the operator
norm $\left\Vert \left(\partial_{0}M\left(\partial_{0}^{-1}\right)+A\right)^{-1}\right\Vert $
is uniformly bounded with respect to $\rho\in\mathbb{R}_{\geq\rho_{0}}$. \end{thm}
\begin{proof}
The result is immediate if we are able to show that \textbf{Condition
(Positivity \ref{Condition-(Positivity I)}) }is satisfied. Due to
propositions \ref{pro:Apos} and \ref{pro:Asternpos} we only need
to show that for some $\beta_{0}\in\mathbb{R}_{>0}$ we have
\[
\Re\left\langle \chi_{_{\mathbb{R}_{\leq0}}}\left(m_{0}\right)\, U|\partial_{0}M\left(\partial_{0}^{-1}\right)U\right\rangle _{\rho,0,0}\geq\beta_{0}\left\langle \chi_{_{\mathbb{R}_{\leq0}}}\left(m_{0}\right)\, U|U\right\rangle _{\rho,0,0},
\]
\[
\Re\left\langle U|\partial_{0}^{*}M^{*}\left(\left(\partial_{0}^{-1}\right)^{*}\right)U\right\rangle _{\rho,0,0}\geq\beta_{0}\left\langle U|U\right\rangle _{\rho,0,0}
\]
for all $U\in D\left(\partial_{0}\right)=D\left(\partial_{0}^{*}\right)$.
Let us consider the first estimate 
\begin{eqnarray*}
 &  & \Re\left\langle \chi_{_{\mathbb{R}_{\leq0}}}\left(m_{0}\right)\, U|\partial_{0}M\left(\partial_{0}^{-1}\right)U\right\rangle _{\rho,0,0}=\\
 &  & =\Re\left\langle \chi_{_{\mathbb{R}_{\leq0}}}\left(m_{0}\right)\, U|\partial_{0}M_{0}U\right\rangle _{\rho,0,0}+\Re\left\langle \chi_{_{\mathbb{R}_{\leq0}}}\left(m_{0}\right)\, U|M_{1}\left(\partial_{0}^{-1}\right)U\right\rangle _{\rho,0,0},\\
 &  & =\Re\left\langle \chi_{_{\mathbb{R}_{\leq0}}}\left(m_{0}\right)\,\sqrt{M_{0}}U|\partial_{0}\sqrt{M_{0}}U\right\rangle _{\rho,0,0}+\Re\left\langle \chi_{_{\mathbb{R}_{\leq0}}}\left(m_{0}\right)\, U|M_{1}\left(\partial_{0}^{-1}\right)\,\chi_{_{\mathbb{R}_{\leq0}}}\left(m_{0}\right)\, U\right\rangle _{\rho,0,0},\\
 &  & \geq\frac{1}{2}\int_{\mathbb{R}_{\leq0}}\left(\partial_{0}\left|\sqrt{M_{0}}U\right|_{0}^{2}\right)\left(t\right)\;\exp\left(-2\rho t\right)\: dt-\mu_{0}\left|\chi_{_{\mathbb{R}_{\leq0}}}\left(m_{0}\right)\, U\right|_{\rho,0,0}^{2},\\
 &  & \geq\frac{1}{2}\int_{\mathbb{R}_{\leq0}}\left(\partial_{0}\exp\left(-2\rho m_{0}\right)\left|\sqrt{M_{0}}U\right|_{0}^{2}\right)\left(t\right)\; dt+\\
 &  & +\rho\int_{\mathbb{R}_{\leq0}}\left|\sqrt{M_{0}}U\right|_{0}^{2}\left(t\right)\;\exp\left(-2\rho t\right)\: dt-\mu_{0}\left|\chi_{_{\mathbb{R}_{\leq0}}}\left(m_{0}\right)\, U\right|_{\rho,0,0}^{2},\\
 &  & \geq\frac{1}{2}\left|\sqrt{M_{0}}U\left(0\right)\right|_{0}^{2}+\left(\rho\gamma_{0}-\mu_{0}\right)\left|\chi_{_{\mathbb{R}_{\leq0}}}\left(m_{0}\right)U\right|_{\rho,0,0}^{2},\\
 &  & \geq\left(\rho\gamma_{0}-\mu_{0}\right)\left|\chi_{_{\mathbb{R}_{\leq0}}}\left(m_{0}\right)U\right|_{\rho,0,0}^{2}.
\end{eqnarray*}
Similarly we obtain, using $\Re\partial_{0}^{*}=\Re\partial_{0}=\rho$,
\begin{eqnarray*}
\Re\left\langle U|\partial_{0}^{*}M^{*}\left(\left(\partial_{0}^{-1}\right)^{*}\right)U\right\rangle _{\rho,0,0} & = & \Re\left\langle \sqrt{M_{0}}U|\partial_{0}^{*}\sqrt{M_{0}}U\right\rangle _{\rho,0,0}+\Re\left\langle U|M_{1}^{*}\left(\left(\partial_{0}^{-1}\right)^{*}\right)\, U\right\rangle _{\rho,0,0}\\
 & \geq & \rho\left\langle \sqrt{M_{0}}U|\partial_{0}\sqrt{M_{0}}U\right\rangle _{\rho,0,0}-\mu_{0}\left|\chi_{_{\mathbb{R}_{\leq0}}}\left(m_{0}\right)\, U\right|_{\rho,0,0}^{2}\\
 & \geq & \left(\rho\gamma_{0}-\mu_{0}\right)\left|U\right|_{\rho,0,0}^{2}.
\end{eqnarray*}
Thus \textbf{Condition (Positivity \ref{Condition-(Positivity I)})
}is satisfied with 
\[
\beta_{0}\;:=\rho_{0}\gamma_{0}-\mu_{0}
\]
where $\rho_{0}\in\mathbb{R}_{>0}$ is chosen such that
\[
\rho_{0}>\frac{\mu_{0}}{\gamma_{0}}.
\]

\end{proof}

\section{Conclusion}

We have presented an extension of a Hilbert space approach to a class
of evolutionary problems. As an illustration we have applied the general
theory to a particular problem concerning the propagation of acoustic
waves in complex media with a dynamic boundary condition. 
\selectlanguage{english}%

\end{document}